\documentclass{amsart}

\usepackage{amsmath,amssymb,amsthm}
\usepackage{mathrsfs}
\usepackage{amscd}
\usepackage{enumerate}
\usepackage{pdfsync}

\usepackage{graphicx}

\theoremstyle{plain}
\newtheorem{theorem}{Theorem}[section]
\theoremstyle{remark}
\newtheorem{remark}[theorem]{Remark}

\newtheorem{definition}[theorem]{Definition}
\theoremstyle{plain}
\newtheorem{corollary}[theorem]{Corollary}
\newtheorem{lemma}[theorem]{Lemma}
\newtheorem{proposition}[theorem]{Proposition}

\numberwithin{equation}{section}

\def\N{{\mathbb N}}

\def\R{{\mathbb R}}
\def\C{{\mathbb C}}

\def\mH{\mathcal{H}}
\def\mT{SIO}
\def\CZK{SK}
\def\mB{\mathcal{B}}
\def\mM{\mathcal{M}}

\usepackage{color}

\definecolor{gr}{rgb}   {0.,   0.8,   0. }
\definecolor{bl}{rgb}   {0.,   0.5,   1. }
\definecolor{mg}{rgb}   {0.7,  0.,    0.7}

\newcommand{\Bk}{\color{black}}

\makeatletter
\@namedef{subjclassname@2010}{
  \textup{2010} Mathematics Subject Classification}
\makeatother

\title{Singular integral operators on tent spaces}

\author{Pascal Auscher, Christoph Kriegler, Sylvie Monniaux, Pierre Portal}

\date{revised, March 15, 2012}

\thanks{The second author acknowledges financial support from the Karlsruhe House of Young Scientists (KHYS)}

\begin{document}

\maketitle

\begin{abstract}
We extend the recent results  concerning boundedness of the maximal regularity operator on tent spaces. This leads us to develop a singular integral operator theory on tent spaces. Such operators have operator-valued kernels.  A seemingly appropriate condition on the kernel is time-space decay measured by off-diagonal estimates {with various exponents}.
\end{abstract}

\section{Introduction}

Let $-L$ be a densely defined closed linear operator acting on $L^{2}(\R^{n})$ and generating
a bounded analytic semigroup $(e^{-tL})_{t \geq 0}$.
Consider the maximal regularity operator originally defined for $f \in L^2(\R_{+},dt; D(L))$, $\R_{+}=(0,+\infty)$, by the Bochner  integral
\begin{equation}
\label{eq:maxop}
\mathcal{M}_{L}f(t) = \int \limits _{0} ^{t} Le^{-(t-s)L}f(s)\, ds.
\end{equation}
This is an example of a singular integral operator with operator-valued kernel.
The bounded extension of this operator to $L^{2}(\R_{+},dt; L^2(\R^n))=L^2(\R^{n+1}_{+}, dtdx)$, $ \R^{n+1}_{+}= \R_{+}\times \R^n$, was established by
de Simon in \cite{desimon}. {The maximal regularity operator plays a crucial role in evolution equations, where its boundedness is used to deduce a priori estimates, which, in turn, can be used to solve non-autonomous and/or non linear problems (see the lecture notes \cite{kuw}). It has thus been the source of intense study, especially in the past 10 years, in $L^p$, and in Besov spaces. As part of the recent development of an evolution equation approach to boundary value problems
on the upper half-space with data in $L^2(\R^n)$, based on the functional calculus of appropriate Dirac operators, a weighted version of de Simon's theorem is proven in
\cite{aa} and \cite[Theorem~1.3]{aa2}, but can be essentially attributed to the earlier work \cite{hk} (see below).

\begin{theorem} \label{AA} With $L$ as above,    $\mathcal{M}_{L}$ extends to a bounded operator on $ L^{2}( \R^{n+1}_{+}, t^{\beta}dt dx)$ for all $\beta \in (-\infty, 1)$.
\end{theorem}
This was proven before in \cite{ps} for  $\beta \in [0,1)$ and a more general class of operators akin to the ones we introduce next, and then for  $\beta\in (-1, 1)$} in  \cite[Theorem 1.13]{hk} when $L$ has dense range. The range of $\beta$ is shown in \cite{aa} to be optimal. Values such as $\beta=-1$ and also an endpoint result for $\beta=1$ were central for applications to the boundary value problems in \cite{aa}.  It should be noted, however, that while the statement of \cite[Theorem 1.13]{hk} does not include the case $\beta=-1$, its proof via their Proposition 1.14 actually gives Theorem \ref{AA}.
%

The articles \cite{hk} and \cite{ps} actually prove weighted $L^p$ estimates for $1<p<\infty$ and show that weighted maximal regularity is equivalent to the unweighted one. However, the $L^p$ analogue of Theorem \ref{AA} needed in the applications we have in mind does not involve weighted $L^{p}( \R^{n+1}_{+})$ spaces for $p \neq 2$, but more appropriate spaces of functions on the upper half space $\R^{n+1}_{+}$.
Let us explain this fact.

Traditionally, an evolution problem of the form $u_{t}+Lu=g$, with initial value $u_{0}=f \in L^{p}(\R^{n})$, is seen as an ordinary differential equation for $L^{p}(\R^{n})$-valued functions. One assumes that $-L$ generates an analytic semigroup on $L^{p}(\R^{n})$, and looks for maximal regularity in spaces such as $L^{p}(\R_{+};L^{p}(\R^{n}))$.
However, if $L=-\text{div}A\nabla$ is a second order, divergence form elliptic operator on $\R^{n}$ with bounded measurable complex valued coefficients, $-L$ only generates an analytic semigroup on $L^{p}(\R^{n})$ for $p$ in an interval $(p_{-}(L),p_{+}(L))$ including $2$, but not always equal to $(1,\infty)$ (see \cite{memoir}).
In this range, maximal regularity results can be proven using the extrapolation method pioneered by Blunck and Kunstmann in \cite{bk}, and developed in \cite{memoir}. Outside of that range, however, maximal regularity in $L^{p}(\R^{n+1}_{+})$ spaces, weighted or not, cannot hold.
In this paper, we prove maximal regularity results on the (unweighted) tent space $T^{p,2,2}$ for all $p \in (\frac{n}{n+1},\infty]$ (see Proposition \ref{prop:divagrad} below),
even though, for small $p$, $-L$ does not even generate a $C_{0}$-semigroup on $L^{p}(\R^{n})$.

Moreover, even when $L=-\Delta$, the free evolution $(t,y) \mapsto e^{t \Delta} f(y)$ does not belong to $L^{p}(\R^{n+1}_{+})$  when $f\in L^p(\R^n)$. This can be compensated by assuming more regularity on $f$, or by using a weighted $L^{p}(\R^{n+1}_{+})$ space with an appropriate weight. However, when dealing with $L^p$ initial data (in boundary value problems, or evolution problems with rough data, for instance), it is desirable to use a norm of the heat extension $(t,y) \mapsto e^{t \Delta} f(y)$ that is equivalent to the $L^p$ norm of $f$ for $p \in (1,\infty)$, and to its $H^p$ norm for $p\in (0,1]$. Weighted $L^p(\R^{n+1}_{+})$ norms do not have this property, but classical harmonic analysis gives many different norms that do.


The one which is of interest to us is given by the following area integral:
$$
\|f\|_{p}\eqsim   \bigg(\ \int \limits _{\R^{n}}
\Bigl( \ \ \iint \limits  _{\R^{n+1}_{+}} \frac{1_{B(x,t^{\frac{1}{2}})}(y)}{t^{\frac{n}{2}}} \,
\bigl|\Delta e^{t\Delta}f(y)\bigr|^{2} \,tdtdy \Bigr)^{\frac{p}{2}}dx \bigg)^{\frac{1}{p}}.
$$
Such a characterisation of the $L^p$ (or $H^p$) norm of a function in terms of its heat extension originates from the work of Fefferman-Stein \cite{fs}.
In more {recent terminology}, this says that $\Delta e^{t\Delta}f$ belongs to a parabolic version of {one of the tent spaces introduced}  by Coifman-Meyer-Stein \cite{cms}.

Now, if one considers the ``mild solution'' $u$ of  $u_{t}-\Delta u=g$ and $u_{0}=0$, given formally  by the
integral formula $\int\limits_{0}^t e^{(t-s)\Delta} g(s)\, ds$, one is led to consider the  boundedness  of the maximal regularity operator $\mathcal{M}_{-\Delta}$ in the norm above. Having such a priori estimates in the same space as the free evolution {$(t,y) \mapsto e^{t\Delta}f(y)$} is a first step towards solving, for example, non-linear problems with $L^p$ data.  Remark that this solution space has, a priori, nothing to do with the space of continuous functions of $t$ with values in $L^p$.
{We thus depart from the tradition of looking at evolution problems for functions on $\R^{n+1}_{+}$ as Banach space valued ODE, and work on spaces where the time and space variables are intrinsically linked.
}
We refer to   { \cite{anp}, and the forthcoming \cite{amp2},}  for more on the PDE aspect of such questions via a tent space approach. We just mention here that this idea goes back (at least) to Koch-Tataru's work on Navier-Stokes equations \cite{kt}.

We introduce the alluded variants of the tent spaces as follows. For $0<p<\infty$, $m\in \N^*$, $\beta\in \R$, define the tent space $T^{p,2,m}(t^\beta\, dtdy)$ as the space of all locally square integrable functions on $ \R^{n+1}_+$  such that
$$
\|g\|_{T^{p,2,m}(t^{\beta}dt dy)} = \bigg(\ \int \limits _{\R^{n}}
\Bigl( \ \ \iint \limits  _{\R^{n+1}_{+}} \frac{1_{B(x,t^{\frac{1}{m}})}(y)}{t^{\frac{n}{m}}} \,
\bigl|g(t,y)\bigr|^{2} t^{\beta}\,dtdy \Bigr)^{\frac{p}{2}}dx \bigg)^{\frac{1}{p}}<\infty.
$$
The classical case is $\beta=-1$, $m=1$, in which case, the space is simply denoted by $T^{p,2}$.
{Since $\|g\|_{T^{p,2,m}(t^{\beta}dt dy)} = \|\tilde{g}\|_{T^{p,2}}$, where $\tilde{g}(s,y) = \sqrt m\, g(s^{m},y)s^{\frac{m(\beta+1)}{2}}$,
${T^{p,2,m}(t^{\beta}dt dy)}$ is  isometric to $T^{p,2}$.}
{However, the parameter $m$ is needed to handle different homogeneities (corresponding to differential operators of different orders), and the parameter $\beta$ is used to handle different applications (e.g. different degree of smoothness for initial data in evolution problems).}
We also remark that a simple use of Fubini's theorem shows that
$\|g\|_{T^{2,2,m}(t^{\beta}dt dy)}^2 = b_{n} \|g\|^2_{L^{2}( \R^{n+1}_+, t^{\beta}dt dy)}$,  whatever the parameter $m$ is, with $b_{n}$ being the  volume of the Euclidean unit ball.  Therefore, for $p=2$, tent spaces agree with  weighted $L^2$ spaces. But it is easy to show that it is not true when $p\ne 2$.

The nature of the norm (a quasi-norm when $p<1$), makes local square integrability {a requirement}. As already showed in \cite{amr} (and subsequently in \cite{hnp}) for different types of operators, a pertinent notion for boundedness of the maximal regularity operator on tent spaces  is a measure of decay of the semigroup called ($L^2-L^2$) off-diagonal estimates.

\begin{definition}\label{def:offdiagonal}
A family of bounded linear operators $(T_{t})_{t \geq 0} \subset B(L^{2}(\R ^{n}))$ is said to
satisfy off-diagonal estimates of order $M$, with homogeneity $m\in \N^*$, if, for all Borel sets
$E,F \subset \R^{n}$, all $t>0$, and all $f \in L^{2}(\R^{n})$:
$$
\|1_{F}T_{t}1_{E}f\|_{2} \lesssim \Big(1+\frac{dist(E,F)^{m}}{t}\Big)^{-M}\|1_{E}f\|_{2}.
$$
Here, and in what follows, $\|\cdot \|_{2}$ denotes the norm in $L^2(\R^n)$.
\end{definition}
{
This property is a replacement for pointwise kernel estimates, which is satisfied, for instance, by heat semigroups generated by elliptic operators with rough coefficients. {Note that we allow a polynomial decay.}}

With  the definition above, the following result  was proved in \cite{amp}.

\begin{theorem}
\label{thm:main}
Let $m \in \N^*$, $\beta \in (-\infty,1)$,
 $p \in \bigl(\frac{2n}{n+m(1-\beta)},\infty\bigr) \cap (1,\infty)$, and $\tau=\min(p,2)$. If
 $(tLe^{-tL})_{t \geq 0}$ satisfies off-diagonal estimates of order
$M>\frac{n}{m\tau}$, with homogeneity $m$,  then $\mathcal{M}_{L}$ extends to a bounded operator on
$T^{p,2,m}(t^{\beta}dt dy)$.
\end{theorem}

The surprise is to obtain results for $p<2$. {
This is particularly true in applications to stochastic parabolic PDEs. Results in this context have been developed in parallel to this article in \cite{anp}, which contains lighter versions of some of the material presented here. In the present paper}
we concentrate more on the abstract theory, and try to weaken  assumptions as much as possible.  This is important even for maximal regularity operators, see Section \ref{sec:max}.

An end-point result, for $p=\infty$, was also obtained in \cite{amp}. In this context, the appropriate tent space consists of
functions such that $|g(t,y)|^{2}\frac{dtdy}{t}$ is a Carleson measure, and is defined as the space of all locally square integrable functions  such that
$$
\|g\|_{T^{\infty,2}}^2=
\underset{(x,r) \in \R^{n} \times \R_{+}}{\sup}
r^{-n} \int \limits _{B(x,r)} \int \limits _{0} ^{r} |g(t,y)|^{2} \frac{ dtdy}{t}<\infty.
$$
The weighted version {(defined through a change of variable as above) is given by}
$$
\|g\|_{T^{\infty,2,m}(t^{\beta}dtdy)} ^{2} :=
\underset{(x,r) \in \R^{n} \times \R_{+}}{\sup}
r^{-{n}} \int \limits _{B(x,r)} \int \limits _{0} ^{r^m} |g(t,y)|^{2} t^{\beta} \,dtdy.
$$

\begin{theorem}
\label{thm:infini}
Let $m \in \N^*$ and $\beta \in (-\infty,1)$.
If $(tLe^{-tL})_{t \geq 0}$ satisfies off-diagonal estimates of order $M>\frac{n}{2m}$, with
homogeneity $m$, then $\mathcal{M}_{L}$
extends to a bounded operator on $T^{\infty,2,m}(t^{\beta}dtdy)$.
\end{theorem}

Note that the backward maximal regularity operator
$$
\mathcal{M}_{L}^-f(t) = \int \limits _{t} ^{\infty} Le^{-(s-t)L}f(s)\, ds,
$$
can be studied on tent spaces, either by duality as $\mathcal{M}_{L}^-= (\mathcal{M}_{L^*})^*$, or directly.

Here, {we continue the development of such tent space boundedness} {results,} and we obtain three-fold improvements. The main statements are in the core of the article. We give here our motivation and extract sample {new} results as  illustrations.

The first observation is that  the conclusion of Theorem \ref{thm:main} is far from optimal in concrete situations.
For instance, for $-\Delta$ (heat semigroup), and its square root  $\sqrt{-\Delta}$ (Poisson semigroup), or even  $-\Delta +V$ with $V\in L^1_{loc}(\R^n)$, $V\ge 0$, and its square root,  or  $ -{\rm div} \,A\nabla$  a second order divergence form elliptic operator on $\R^n$ with bounded,
measurable, \emph{real}-valued coefficients, and  its square root,  the range of $p$ can be much improved.

\begin{proposition}
\label{cor:Laplace}
\begin{enumerate}
  \item $\mathcal{M}_{-\Delta +V}$ and   $\mM_{-{\rm div} \,A\nabla}$ (with real-valued coefficient {matrix} $A$) extend to bounded operators  on
 $T^{p,2,2}(dtdy)$ when $\frac{n}{n+1}<p\le \infty$.
  \item  $\mathcal{M}_{\sqrt{-\Delta +V}}$ and $\mM_{\sqrt{-{\rm div} \,A\nabla}}$ (with real-valued coefficient {matrix} $A$)  extend   to  bounded operators on
 $T^{p,2,1}(t^{-1} dtdy)$ when $\frac{n}{n+1}<p\le \infty$.

\end{enumerate}
    \end{proposition}



 {The range of $p$ is a consequence of the pointwise decay of the corresponding heat kernels.}
 However, not all semigroups obey pointwise decay.  In that case, one can use intermediate conditions  between pointwise decay and $L^2-L^2$ off-diagonal estimates such as     $L^q-L^r$ off-diagonal estimates with $q\le r$ and $q=2$ or $r=2$    (see Definition \ref{def:qroff}). This information can {then} be used to quantify the range of $p$ {for tent space boundedness}.
This is the case for  $ -{\rm div} \,A\nabla$ with  complex-valued coefficients.  Here, the decay is Gaussian but the range of $q$ or $r$ may be limited as dimension increases.

\begin{proposition}\label{prop:divagrad}  For a complex-valued coefficient {matrix} $A$,   $\mM_{-{\rm div} \,A\nabla}$   extends to a bounded operator on
 $T^{p,2,2}(dtdy)$ when $\frac{1}{2}<p\le \infty$ if $n=1$,  $\frac{2}{3}<p\le \infty$ if $n=2$,  $\frac{6}{7}-\varepsilon<p\le \infty$ if $n=3$, and $2-\frac{4 }{n} -\varepsilon<p\le \infty$ if $n\ge 4$.
 The $\varepsilon >0$ depends on the operator but the lower bound is at least $\frac{n}{n+1}$.
\end{proposition}

{These two propositions (see Section 5 for  their proofs) follow from general statements (proved in Sections 3 and 4) in which one  requires a lower bound  on the polynomial decay exponent $M$ of  Definition \ref{def:qroff}. Note that this lower bound increases with dimension. As the decay here is exponential,  the exponent $M$ can be as large as one wants, and the results apply.

   We now consider the case of polynomial decay. This is our second point. In this case, the value of $M$ is to be compared with the lower bound in our statements for applicability.   For example, one has  $M=1$ in the  $L^2-L^2$ off-diagonal estimates with homogeneity $m=1$ for $\sqrt{-{\rm div} \,A\nabla}$  (even for $\sqrt{-\Delta}$).
 {Theorem \ref{thm:main} requires $M>n/\tau$, but}
one can take advantage of the fact that the exponent $M$ in the $L^q-L^2$ off-diagonal estimates grows linearly in $1/q$ (see  Proposition \ref{prop:qr}).  However, the range of $q$ may be limited as well,  which is the case for ${{-{\rm div} \,A\nabla}}$ operators with  complex-valued coefficients, and again we may not have a large enough value of $M$.

On the other hand, with no decay at all, the $p=2$ boundedness follows from Theorem \ref{AA}. So it seems reasonable to expect a range  of $p$ near 2 depending on $q$ and $M$, when $q\sim 2$ and $M>0$ is small, by some kind of interpolation procedure.  We will obtain (see Section 5) a general result in this direction, which  gives, for this  particular operator, the proposition below.}

\begin{proposition}\label{prop:sqrtdivagrad}  For a complex-valued  coefficient matrix $A$,
  $\mM_{\sqrt{-{\rm div} \,A\nabla}}$   extends to a bounded operator   on
 $T^{p,2,1}(t^{-1}dtdy)$ when $\frac{1}{2}<p\le \infty$ if $n=1$,  $\frac{2}{3}<p\le \infty$ if $n=2$,  $\frac{6}{7}-\varepsilon<p\le \infty$ if $n=3$,  $1 -\varepsilon<p\le \infty$ if $n= 4$ and  $2-\frac{4 }{n} -\varepsilon < p < \frac{2n-4}{n-4}+\varepsilon'$ if $n\ge 5$. The $\varepsilon, \varepsilon' >0$ depend on the operator but the lower bound is at least $\frac{n}{n+1} $. \end{proposition}

{To do this interpolation procedure,  we view the maximal regularity operator within a family of operators {of the same nature}.  Thus,  and this is the third point,  it becomes interesting and convenient  to  develop an abstract formulation that is not restricted to the maximal regularity operator. We introduce, in the next section, a class of  singular integral operators  in the context of tent spaces.  Sufficient conditions for their boundedness are given in Sections 3 and 4. We remark that, in contrast to the usual $L^p$ theory for  Calder\'on-Zygmund  operators, no regularity of the kernel is necessary.  In a sense, despite the fact that tent spaces, for $1<p<\infty$, can be seen as subspaces of Hilbert-valued L$^p$ spaces (\cite{htv}), Calder\'on-Zygmund theory does not seem to be an appropriate machinery to study singular integral operators in this context.
We depart from the usual treatment of  maximal regularity {through} a singular integral operator acting on some Banach-valued functions. Here, we start from the ``easy'' Hilbert  ($L^2$) space  theory, and then move on to tent spaces, using the notion of $L^q-L^r$ off-diagonal decay, which extends the notion of  $L^2-L^2$ off-diagonal decay defined above.}

{
\begin{remark} Our results can, nevertheless, be extended to the context of tent spaces of Banach space-valued functions (provided the Banach space $X$ is UMD, and $1<p<\infty$). This is done by adapting the arguments of \cite{hnp} to take advantage of $L^q-L^2$ (for $q \leq 2$, resp. $L^2-L^q$ for $q \geq 2$), rather than $L^2-L^2$, off-diagonal estimates, in the same way it is done in this paper. However, the obvious adaptation does not seem to produce optimal relationships between $p,q,M$, and the geometry of $X$. We choose not to attempt to address this issue here.
\end{remark}
}

\section{Singular integral operators}

\subsection{Abstract setup}

Consider a separable complex Hilbert space $H$. For $\beta\in \R$, set $\mH_{\beta}= L^2(\R_{+}, t^\beta dt;  H)$ . We consider the following classes of operators $\mT^\pm \subset \mB(\mH_{0})$.

\begin{definition}
\begin{enumerate}
\item We say $T \in \mT^+$ if
$T \in  \mB(\mH_{0})$ and there exist  a {strongly} measurable family of operators $K(t,s)\in \mB(H)$, $t,s\in \R_{+}$ and a constant $C<\infty$ such that
$\|K(t,s)\| \le C |t-s|^{-1}$ and
\begin{equation}
\label{eq:rep+}
 Tf(t) = \int\limits_{0}^t K(t,s)f(s)\, ds
\end{equation}
 for all $f\in \mH_{0}$ with bounded support in $\R_{+}$ and  almost all $t\in \R_{+}$ not in the support of $f$.
 \item  We say $T \in \mT^-$ if
$T \in  \mB(\mH_{0})$ and there exist  a {strongly} measurable family of operators   $K(t,s)\in \mB(H)$, $t,s\in \R_{+}$  and a constant $C<\infty$  such that
$\|K(t,s)\| \le C |t-s|^{-1}$ and $T$ has the representation
 \begin{equation}
\label{eq:rep-}
Tf(t) = \int\limits _{t}^\infty K(t,s)f(s)\, ds
\end{equation}
 for all $f\in \mH_{0}$ with bounded support  in $\R_{+}$ and  almost all $t\in \R_{+}$ not in the support of $f$.
\end{enumerate}

\end{definition}   We remark that $K(t,s)$ need only be defined on $s<t$ for $T\in \mT^+$ and on $t<s$ for $T\in\mT^-$ and the value at $t=s$ is irrelevant.  With this precaution, we say that $T\in \mT^\pm$ is associated to the operator-valued kernel $K(t,s)$ {and that such kernels belong to the class  $\CZK^\pm$ of singular kernels}.

 Our terminology follows in part  that of singular integrals (here with operator-valued kernels) but we assume a sign condition on $s-t$.

 Let us make a few remarks.

 The representation   \eqref{eq:rep+}  of $Tf$ above is {a Bochner integral} and the equality holds in $H$.  It  is  clearly equivalent to
 \begin{equation}
\label{eq:rep+2}
\langle Tf, g\rangle = \iint\limits_{s<t} \langle K(t,s)f(s), g(t)\rangle dsdt
\end{equation}
 for $f,g\in \mH_{0}$ having bounded {disjoint} support.
 The inner product on the left is the canonical one in $\mH_{0}$, and on the right the canonical one in $H$.

  It is clear that  $T\in \mT^+$ if and only if $T^*\in \mT^-$, with associated kernel $K(s,t)^*$. Hence, similar comments apply to \eqref{eq:rep-}.

 The basic examples are of course $\mM_{L}\in \mT^+$ and $ \mM_{L}^-\in \mT^-$.  For $\mM_{L}$ the boundedness on $\mH_{0}$ is given by de Simon's theorem. Then the formula \eqref{eq:rep+2}  holds for all $f\in L^2(\R_{+}, dt; D(L)) $ and all $g\in \mH_{0}$ with continuous kernel $K(t,s)=Le^{-(t-s)L}$ on $s<t$. If now, $f,g$ have disjoints supports, one can argue by density  of $D(L)$ in $H$.  For $\mM_{L}^-$, we simply use $\mM_{L}^-=(\mM_{L^*})^*$.

There is a natural splitting of  operators $T \in \mT^+$ into an integral part plus a singular part.    Let $K$ be the associated kernel. Using that $t-s \sim t$ when $s<t/2$ and Hardy inequality, one has
 $$
 \int\limits_{0}^\infty \Bigg( \int\limits_{0}^{\frac t 2}\|K(t,s)\| \|f(s)\|\,ds \Bigg)^2\, dt \lesssim \int\limits_{0}^\infty \Bigg( \frac{1}{t}\int\limits_{0}^{\frac t 2} \|f(s)\|\,ds \Bigg)^2\, dt \lesssim \|f\|^2_{\mH_{0}}.
 $$
 Hence, the integral part of $T$ is the  operator defined for  $f\in \mH_{0}$  for almost all $t>0$ by the  {Bochner} integral
 $$(T_{2}f)(t)= \int \limits _{0} ^{\frac{t}{2}} K(t,s)f(s)\, ds,$$
and $T_{2} \in \mT^+$ as well. The singular part is  $T_{1}:=T-T_{2} \in \mT^+$,  and  carries  the singularity at $s=t$. Its associated kernel is $K(t,s) 1_{t/2<s<t}$.
Note that, for the integral part, the integral representation is  {valid without restriction on $f$ and $t$}.

For $T\in \mT^-$, one has the same splitting with  $T_{2}f(t)= \int \limits^\infty _{2t} K(t,s)f(s)\, ds$ as the integral part, and $T_{1}=T-T_{2}$ as the singular part.

Theorem \ref{AA} and its proof carry to this abstraction.

 \begin{theorem} \label{AAabs}  Let  $\beta \in (-\infty, 1)$. Any operator in $T\in\mT^+$  extends to a bounded operator on $\mH_{\beta}$ which is denoted by $T$ as well. Furthermore, for any kernel $K\in \CZK^+$ and  $f\in \mH_{\beta}$,  $ \int \limits _{0} ^{\frac{t}{2}} \|K(t,s)\|\|f(s)\|\, ds $ is an element of $L^2(\R_{+}, t^\beta dt)$, so that  for almost all $t>0$,
  $ \int \limits _{0} ^{\frac{t}{2}} K(t,s)f(s)\, ds $   is a Bochner integral   in $H$. {If, in particular, $K$ is the kernel of $T$ then this integral   agrees} with $(T_{2}f)(t)$.

 The same statement holds for $T\in \mT^{-}$ and $-\beta\in (-\infty,1)$.
\end{theorem}

We include a quick argument. For $\alpha=\beta/2<1/2$, we have that
 $$
\int\limits_{0}^\infty \bigg(\int\limits_{0}^t \|K(t,s)\| \, |t^\alpha-s^\alpha|\,  {s^{-\alpha}}\|{s^{\alpha}}f(s)\|\, ds\bigg)^2 \, dt \lesssim \|f\|_{\mH_{\beta}}^2
$$
using the Schur test and the bound on $K$. Hence, the integral
operator { $f\mapsto [t \mapsto \int\limits_{0}^t K(t,s) (t^\alpha-s^\alpha) f(s)\, ds]$  is
bounded from $\mH_{\beta}$ to $\mH_{0}$. For $f\in \mH_{\beta}$ with compact support in $\R_{+}$, it agrees with
$t^\alpha({T}f)(t)- (T(s^\alpha f))(t)$. Since $T\in B(\mH_{0})$, this readily gives the result by density.}

The second part follows from the weighted Hardy inequalities {\cite{mu}} when $\beta<1$
$$
\int\limits_{0}^\infty \Bigg( \frac{1}{t}\int\limits_{0}^{\frac t 2} \|f(s)\|\,ds \Bigg)^2\, t^\beta dt \lesssim \|f\|^2_{\mH_{\beta}}.
$$

The proof for $\mT^-$ is left to the reader.

\subsection{Concrete situation}

Now, in order to get tent space results, we specialise to $H=L^2(\R^n)$, and  introduce subclasses.  First recall that $\mH_{\beta}$ can be identified with $L^2( \R^{n+1}_+ ,t^\beta dtdy)$. Hence, we now write $f(s)$ as $f$ or  $f(s,\cdot)$ if we want to specialise the $s$ variable.  Using that we have a spatial variable, we extend \eqref{eq:rep+}   as follows.

 \begin{lemma}\label{lem:rep} Let $\beta<1$ and $T\in \mT^+$. Let $E,F$ be two  Borel sets of $\R^n$ and $I,J$ two open intervals in $\R_{+}$. Assume that  $f \mapsto [ (t,y) \mapsto \int\limits_{0}^t |(K(t,s)f(s, \cdot))(y)|\, ds]$ is bounded from the space of functions  $f\in L^2(s^\beta dsdx)$ with  support in $I\times E$  into $L^2(J\times F, t^\beta dtdy)$.
 Then the representation $Tf(t,y)= \int\limits_{0}^t (K(t,s)f(s,\cdot))(y)\, ds$ holds  for all such $f$ {with equality in $L^2(J \times F, t^\beta dtdy)$.}

 The corresponding statement holds for $T\in \mT^-$ and $-\beta<1$.
\end{lemma}
Remark that this lemma is only needed for singular parts. For regular parts, the representation is valid without support conditions.

\begin{proof} Both terms are defined in $L^2(J\times F, t^\beta dtdy)$ by assumption so that it suffices to prove  the following claim:
$$
\langle Tf,g\rangle =\iint\limits_{J\times F} \bigg(\int\limits_{0}^t (K(t,s)f(s,\cdot))(y)\, ds\bigg)\,  \overline g(t,y) dtdy
$$
for all $f \in C^\infty_{0}(I; L^2(E))$ and $g \in C^\infty_{0}(J; L^2(F))$.  We implicitly extend $f(s,\cdot)$ by 0 outside $E$ and $g(t,\cdot)$ by 0 outside $F$. Remark that, from the assumption,
$(s,t,y) \mapsto (K(t,s)f(s,\cdot))(y) \overline g(t,y) 1_{s<t}$ is integrable with integral bounded by $\|f\|_{L^2(s^\beta dsdx)}\|g\|_{L^2(t^\beta dtdy)}$, hence, by Fubini's theorem, we only have to show
$$
\langle Tf,g\rangle=\iint\limits_{s<t} \langle K(t,s)f(s,\cdot), g(t,\cdot)\rangle \,  dsdt.
$$
Choose orthonormal bases $(e_{j})$ of $L^2(E)$ and $(\varepsilon_{k})$ of $L^2(F)$.
 By a limiting  argument for each term,
 it is enough to assume that $f(s,\cdot)$ and $g(t,\cdot)$ take values  in finite dimensional linear spans of the respective bases. Indeed, use boundedness of $T$ in the left hand side and the integrability assumption in the right hand side.  By linearity, it is enough to assume that $f(s,\cdot)=  f_{j}(s)  e_{j}$ and $g(t,\cdot)=  g_{k}(t) \varepsilon_{k}$
 for scalar test functions $f_{j}, g_{k}$.  In this case, there is a distribution $S_{j,k}\in \mathcal{D}'(I\times J)$ such that $\langle Tf,g \rangle= (S_{j,k}(t,s),  f_{j}(s)\overline {g_{k}}(t))$. It follows from \eqref{eq:rep+2} and decomposing on the orthonormal bases  that $\langle K(t,s) e_{k}, \varepsilon_{j}\rangle $ is the restriction to $0<s<t<\infty, s\in I, t\in J$ of $S_{j,k}$. Thus the desired equality
 holds for such $f,g$ and we are done.

We skip the similar proof for $T\in \mT^{-}$.
 \end{proof}

 In applications,  it suffices to show (absolute) convergence of the integral $\int\limits_{0}^t K(t,s)f(s, \cdot)\, ds
$ in the norm $L^2(J\times F, t^\beta dtdy)$  to obtain an estimate of $Tf$ in that norm, when $f$ is supported in $I\times E$.   We shall use this when $E$ and $F$ are at  positive distance {and $K(t,s)$ satisfies certain} decay estimates.

{We thus introduce subclasses of $\mT^\pm$, where the size estimate $\|K(t,s)\| \lesssim |t-s|^{-1}$ is  complemented by the following time-space estimates.}

\begin{definition}\label{def:qroff} Let $1\le q\le r\le \infty$.
An operator-valued kernel  $K=(K(t,s))_{t, s > 0} \subset B(L^{2}(\R^{n}))$ is said to
satisfy $L^q-L^r$ decay of order $M> 0$, with homogeneity $m\in \N^*$, if, for all Borel  sets
$E,F \subset \R^{n}$, all $t\ne s$, and all $f \in L^{2}(\R^{n})\cap L^q(\R^n)$:
$$
\|1_{F}K(t,s)1_{E}f\|_{r} \lesssim |t-s|^{-1-\frac{n}{m}(\frac{1}{q}- \frac{1}{r})}\,  \Big(1+\frac{dist(E,F)^{m}}{|t-s|}\Big)^{-M}\|1_{E}f\|_{q}.
$$
Here, and in what follows $\|\cdot \|_{q}$ denotes the norm in $L^q(\R^n)$. \end{definition}

{Note that, in the proofs, one only needs this property for sets of the form $E=B(x,r)$ and $F = B(x,2^{k+1}r)\backslash B(x,2^{k}r)$ (or vice versa). For this restricted property, $L^q-L^r$ decay implies $L^{\tilde{q}}-L^{\tilde{r}}$ decay for $q \leq \tilde{q} \leq \tilde{r} \leq r$ (by H\"older's inequality), but the order of decay changes. See \cite{am} for more on this issue. We do not, however, use this fact in this paper.}

We need only two specific cases: $1\le q \le 2$ and $r=2$, {and $q=2$ and $2\le r\le \infty$.
{
In certain cases, the decay is actually exponential, so the polynomial decay defined above holds for all $M>0$, in which case we say that the order is $\infty$.
In this paper, we are particularly interested in obtaining results under minimal values of polynomial decay.

\begin{definition}\label{def:SIOq} Let $1\le q \le \infty$ and $M\in \R_{+}\cup  \{\infty\}$. We say that $\mathbf{T\in \mT^\pm_{m,q,M}}$ if $T\in \mT^\pm$ and the associated operator-valued kernel $K(t,s)\in \CZK^\pm$ satisfies $L^q-L^2$ (resp. $L^2-L^q$)  decay of order $M$, with homogeneity $m$, when $q\le 2$ (resp. $q\ge 2$).

\end{definition}

The value of $m$ is dictated by the situation, and $q$ and $M$ are the most important parameters. Let us point out that all calculations work with $m$ being any positive real number, rather than just integer. We mention this for potential development towards fractal situations where fractional homogeneity can occur.}

\section{Role of $L^q-L^2$  decay}
\label{sec:result}

{
The range of $p$ below 2 for which $T^{p,2}$ boundedness results hold {can be quantified by}
$L^q-L^2$ decay. Some technical conditions are also required. In particular
the order $M$ should not be too small.
}

 \begin{theorem}\label{thm:lql2abs}   Let $T\in \mT^+_{m,q,M}$  with $1\le q \le 2$, $M>\frac{n}{2m}$ and let $p_{M}<1$ be defined by $M=\frac{n}{2m}(\frac 2 {p_{M}} -1)$.  Let $q'$ be the dual exponent to $q$ and  $\beta<1$.
 \begin{enumerate}
  \item If $q'\le \frac{2n}{m(1-\beta)}$ or equivalently $\frac{n}{2m} \ge -\frac{\beta-1}{2} + \frac{n}{m}(\frac{1}{q}-\frac{1}{2})$
   then
 $T$  extends to a bounded operator on
 $T^{p,2,m}(t^{\beta}dtdy)$ when $2\ge p>p_{c}$, where
 $$p_{c}=\frac{2\left(\frac{n}{2m}-\frac{n}{m}\left(\frac{1}{q}-\frac{1}{2}\right)\right)}{\frac{n}{2m}-\frac{n}{m}\left(\frac{1}{q}-\frac{1}{2}\right)+ \frac{1-\beta}{2}}= \frac{4n}{2n + m(1-\beta)q'}\ge 1\ .
$$
 \item
 If $q'> \frac{2n}{m(1-\beta)}$ or equivalently  $ -\frac{\beta-1}{2} + \frac{n}{m}(\frac{1}{q}-\frac{1}{2})> \frac{n}{2m}$
 then
 $T$  extends to a bounded operator on
 $T^{p,2,m}(t^{\beta}dtdy)$ when $2\ge p>\sup(p_{M},\tilde p_{c})$, where
$$
\tilde p_{c}= \frac{2n}{\frac{2n}{q}+ m(1-\beta)}<1\ .
$$
  \end{enumerate}
\end{theorem}

Let us say a word on the exponents $p_{c}, \tilde p_{c}$. In the first case, $p_{c}\ge 1$. In the second case, $\tilde p_{c}< 1$. It is consistent as
$$
\tilde p_{c} = p_{c} \Longleftrightarrow \tilde p_{c} = 1 \Longleftrightarrow p_{c} = 1   \Longleftrightarrow  \frac{n}{2m} = -\frac{\beta-1}{2} + \frac{n}{m}\Big(\frac{1}{q}-\frac{1}{2}\Big)\ .
$$

{When $q$ is small, we thus get results for $p$ below $1$ provided $M$ is not too small (e.g. in the case of exponential decay).}

As a function of $q$, the exponents $ p_{c},\tilde p_{c}$ are increasing. When $q=2$,
$\tilde p_{c}= \frac{2n}{n+m(1-\beta)}$ which is the exponent found in Theorem \ref{thm:main}.
{ Remark that we improve over the lower bound: $M>\frac n{2m}$ suffices here  instead of $M> \frac n{pm}$ when $p\le 2$. }

{ In \cite{amp}, Theorem \ref{thm:main} was proved using comparison of tent space norms under change of apertures, i.e. $B(x,t^{\frac{1}{m}})$ changed to $B(x,c t^{\frac{1}{m}})$ for $c>1$. The sharp behavior of these comparisons was obtained in \cite{a} using atomic decompositions and interpolation. It is thus natural to use atoms here as well to prove our results. }{Furthermore, it simplifies the proofs greatly.}

Recall  that for $0<p\le 1$, the tent space $T^{p,2}$ has an atomic decomposition \cite{cms}: A $T^{p,2}$ atom is a function $a(t,y)$ supported\footnote{The support is a relatively closed subset of $ \R^{n+1}_+$.} in a region $(0,r]\times B$ where $B$ is a (closed) ball on $\R^n$ of radius $r$, satisfying
$
 \int \limits _{B} \int \limits _{0} ^{r} |a(t,y)|^{2} \frac{ dtdy}{t} \le r^{-n(\frac{2}{p}-1)}.
 $
Any $T^{p,2}$ function $g$ can be represented as a  series $g=\sum \lambda_{j }a_{j}$ where $a_{j}$ is a $T^{p,2}$ atom and $\sum |\lambda_{j}|^p \sim \|g\|_{T^{p,2}}^p$. Here the series converges in the tent space quasi-norm, and, in particular, in $L^2_{loc}(\R^{n+1}_{+})$.
Translating this to our setting, $T^{p,2,m}(t^{\beta}dtdy)$ atoms are functions $A(t,y)$ with support in $(0,r^m]\times B$, where $B$ is a (closed) ball in $\R^n$ of radius $r$,  satisfying
$
 \int \limits _{B} \int \limits _{0} ^{r^m} |A(t,y)|^{2} {t^\beta  dtdy}{} \le r^{-n(\frac{2}{p}-1)},
 $
 and the decomposition theorem holds in $T^{p,2,m}(t^{\beta}dtdy)$. Remark that atoms are also special elements of $L^2( \R^{n+1}_+, t^\beta dtdy)=T^{2,2,m}(t^{\beta}dtdy)$ which is helpful for representation purposes of $\mT^\pm$ operators acting on them.

{
\begin{remark}
\label{rk:iso}
Recall that the map $j:  T^{p,2,m}(t^{\beta}dtdy) \to T^{p,2,1}(t^{-1}dtdy)$ defined by   $j(f)(t,y) =\sqrt m  t^{\frac{m(1+\beta)}{2}}f(t^{m},y)$ is an isometry; it also sends
$T^{p,2,m}(t^{\beta}dtdy)$ atoms to $T^{p,2,1}(t^{-1}dtdy)$ atoms.
\end{remark}
}
\begin{lemma}\label{lem:atomic}
Let $p\le 1$ and $T$ a linear  operator bounded on $T^{2,2,m}(t^{\beta}dtdy)$. Then
$T$ has a bounded extension from $T^{p,2,m}(t^{\beta}dtdy) \cap T^{2,2,m}(t^{\beta}dtdy)$ to
$T^{p,2,m}(t^{\beta}dtdy)$ if it is uniformly bounded on $T^{p,2,m}(t^{\beta}dtdy)$ atoms.
\end{lemma}
 \begin{proof}
Adapt to $p\le 1$ the argument in Step 3 of the proof of Theorem 4.9 in \cite{amr} done for
{
$T^{p,2,1}(t^{-1}dtdy)$ (without loss of generality, one can take $m=1$, and $\beta=-1$ by Remark \ref{rk:iso})}. This argument also furnishes the extension procedure.
\end{proof}
Theorem \ref{thm:lql2abs} follows immediately from the two lemmas below applied to the decomposition of $T \in  \mT^+_{m,q,M}$ into its singular part $T_{1}$ plus   its integral part  $T_{2}$. Recall that $M>\frac{n}{2m}$.

 \begin{lemma}\label{lem:m1}
 The operator $T_{1}$  extends   to
 $T^{p,2,m}(t^{\beta}dtdy)$ for $p>p_{M}$.
\end{lemma}

 \begin{lemma}\label{lem:m2} The statement of Theorem \ref{thm:lql2abs} holds for $T_{2}$.
\end{lemma}

\begin{proof}[Proof of Lemma \ref{lem:m1}] By interpolation { (see \cite{cms} for the case $m=1$, $\beta=-1$, and apply Remark \ref{rk:iso} to deduce the general case)} it suffices to consider $p_{M}<p\le 1$.
By Lemma \ref{lem:atomic}, it is enough to show  that $T_{1} A \in T^{p,2,m}(t^{\beta}dtdy)$ if $A$ is a
$T^{p,2,m}(t^{\beta}dtdy)$ atom, with a uniform bound.
{Since the proofs are scale invariant, we assume that}
$A$ is supported in $(0, 1] \times B(0,1)$. Then we remark that  if $t>2$,  $T_{2}A(t,\cdot)=TA(t,\cdot)$ because of the definition  of  $T_{2}$ and the support of $A$. Hence $(T_{1}A)(t,\cdot)=0$  for $t>2$. We let $f_{j}(t,y)= (T_{1} A)(t,y) $ if $2^j\le |y| <2^{j+1}$, 0 elsewhere, and $f_{0}(t,y)=(T_{1} A)(t,y)$ if $|y|\le 2$, 0 elsewhere. We show that $f_{j}=\lambda_{j}A_{j}$ with $A_{j}$ a $T^{p,2,m}(t^{\beta}dtdy)$ atom and  $\sum |\lambda_{j}|^p \lesssim 1$.

For $j=0$, this follows from the boundedness of $T_{1}$ on $T^{2,2,m}(t^{\beta}dtdy)$ as $\beta<1$. For $j\ge 1$, we  argue as follows:
\begin{align*}
 \int\limits_{B(0,2^{j+1})} \int \limits _{0} ^{2^{(j+1){m}}}& |f_{j}(t,y)|^2 t^{\beta}\,dtdy
\\
&=  \int \limits _{0} ^{2} \int\limits_{2^j\le |y| <2^{j+1}} |(T_{1}A)(t,y)|^2\, dy \,  t^{\beta}\,dt
\\
&{=   \int \limits _{0} ^{2}\int\limits_{2^j\le |y| <2^{j+1}}  \bigg\vert\int \limits _{\frac{t}{2}} ^{t} \bigg(\frac{t-s}{t-s}\bigg)^{\varepsilon-\frac{1}{2}} (K(t,s) A(s,\cdot))(y)  \, ds\bigg\vert^2dy \,  t^{\beta}\,dt
}\\
& \lesssim   \int \limits _{0} ^{2}\int\limits_{2^j\le |y| <2^{j+1}}  \int \limits _{\frac{t}{2}} ^{t} t^{2\varepsilon}(t-s)^{1-2\varepsilon  } |(K(t,s) A(s,\cdot))(y)|^2  \, dsdy \,  t^{\beta}\,dt
\\
& \lesssim   \int \limits _{0} ^{2} \int \limits _{\frac{t}{2}} ^{t} t^{2\varepsilon}\frac 1{(t-s)^{1+2\varepsilon +\frac{2n}{m}(\frac{1}{q}-\frac{1}{2})} }\Bigl(1+\frac{2^{jm}}{t-s}\Bigr)^{-2M}
\|A(s,.)\|_{q} ^{2}\, t^{\beta}\, ds\,dt
\\
& \lesssim
 \int \limits _{0} ^{1} \|A(s,.)\|_{2} ^{2} s^{\beta}   s^{2\varepsilon} \int \limits _{s} ^{2s}\frac 1{(t-s)^{1+2\varepsilon + \frac{2n}{m}(\frac{1}{q}-\frac{1}{2})}}
\Bigl(1+\frac{2^{jm}}{t-s}\Bigr)^{-2M}\, dt  ds
\\
& \lesssim    2^{-2jmM} \int \limits _{0} ^{1}  \|A(s,.)\|_{2} ^{2}\, s^{\beta}ds.
\end{align*}
{We used Cauchy-Schwarz inequality in the fourth line and $t^{2\varepsilon}\eqsim  \int  _{\frac{t}{2}} ^{t} {(t-s)}^{2\varepsilon-1}\, ds$ when $\varepsilon>0$. In the next to  last line, we impose }
 {$\varepsilon<M- \frac{n}{m}(\frac{1}{q}-\frac{1}{2})$}, which is possible as
$M>\frac{n}{2m} $ and $q\ge 1$.
The estimate $\|A(s,.)\|_{q}  \lesssim \|A(s,.)\|_{2}$  uses the fact that $A(s,\cdot)$ is supported in $B(0,1)$. As $\gamma=2mM - n(\frac{2}{p}-1)>0$, we thus get the desired estimate with $\lambda_{j}=C2^{-j\gamma/2}$. We also remark that we implicitly used Lemma \ref{lem:rep}, {which is possible since} the last four lines yield the required estimate to write $T_{1}A(t,y)=\int\limits_{\frac{t}{2}}^t (K(t,s)A(s,\cdot))(y)\, ds$ on the support of $f_{j}$.
\end{proof}

\begin{proof}[Proof of Lemma \ref{lem:m2}] We imbed $T_{2}$ into an analytic family of integral operators $\mathcal{J}_{\alpha}$ defined  for $\alpha\in \C$ by
$$
\mathcal{J}_{\alpha}f(t,y)= \int \limits ^{\frac t 2} _{{0}{}} \left(\frac{s}{t}\right)^\alpha (K(t,s)f(s,\cdot))(y)ds.
$$
Observe that
$$
\iint\limits_{\R^{n+1}_{+}}|\mathcal{J}_{\alpha}f(t,y)|^2 t^\beta dtdy=  \iint\limits_{\R^{n+1}_{+}}\bigg|
 \int \limits ^{\frac t 2} _{{0}{}} \bigg(\frac{s}{t}\bigg)^{\alpha- \frac{\beta-1}{2}} (tK(t,s)(s^{\frac{\beta+1}{2}}f(s,\cdot)))(y)\frac{ds}{s}\bigg|^2\, \frac{dtdy}{t}.
 $$
 An application of Schur's lemma, using that $t\sim t-s$ and the uniform boundedness of $tK(t,s)$,
 shows that, provided $ \Re e\,  \alpha  - \frac{\beta-1}{2}>0$, the last integral is bounded by
 $$ C\bigg( \Re e\, \alpha  - \frac{\beta-1}{2}\bigg)  \iint\limits_{\R^{n+1}_{+}}
 |s^{\frac{\beta+1}{2}}f(s,x)|^2\, \frac{dsdx}{s}= C\bigg( \Re e\, \alpha  - \frac{\beta-1}{2}\bigg)  \iint\limits_{\R^{n+1}_{+}}
 |f(s,x)|^2\, s^\beta{dsdx}.
 $$
 Hence, $\mathcal{J}_{\alpha}$ is well-defined for $ \Re e\, \alpha  > \frac{\beta-1}{2}$ and bounded on $T^{2,2,m}(t^{\beta}dtdy)$ for all $m$. Notice that $\beta<1$ implies that this domain contains $\alpha=0$ and $\mathcal{J}_{0}= T_{2}$.

 Now we let $A$ be a $T^{p,2,m}(t^{\beta}dtdy)$ atom and estimate $\mathcal{J}_{\alpha}A$.   {Since the proof below is scale invariant,} we assume that $A$ is supported in $(0,1]\times B(0,1)$. We let
{$$ f_{j}(t,y)= \begin{cases}
(\mathcal{J}_{\alpha} A)(t,y) \; \text{if} \; 2^j\le |y| <2^{j+1} \; \text{and} \; t<2^{jm},\\
(\mathcal{J}_{\alpha} A)(t,y) \; \text{if} \; |y| <2^{j+1} \; \text{and}\; 2^{jm}\le t < 2^{(j+1)m},\\
0 \; \text{otherwise},
 \end{cases}
 $$
 }
{for $j\ne 0$ and $f_{0}(t,y)=(\mathcal{J}_{\alpha} A)(t,y)$ if $|y|\le 2$ and $t<2^m$, 0 elsewhere, so that $\mathcal{J}_{\alpha} A= f_{0}+f_{1}+\ldots$}

By the boundedness property of $\mathcal{J}_{\alpha}$, we get
$$
\int\limits_{B(0,2)}\int\limits_{0}^{2^m}|f_{0}(t,y)|^2 t^\beta dtdy \le
C\bigg( \Re e\, \alpha  - \frac{\beta-1}{2}\bigg) \int\limits_{B(0,1)}\int\limits_{0}^{1}|A(s,x)|^2 s^\beta dsdx
\le C\bigg( \Re e\, \alpha  - \frac{\beta-1}{2}\bigg).
$$

 Next,
 \begin{align*}
 \int\limits_{B(0,2^{j+1})}\int\limits_{0}^{2^{(j+1)m}}|f_{j}(t,y)|^2 t^\beta dtdy
 &= \int\limits_{2^j<|y|<2^{j+1}}\int\limits_{0}^{2^{jm}}|f_{j}(t,y)|^2 t^\beta dtdy \\
 \\& \qquad + \int\limits_{|y|<2^{j+1}}\int\limits_{2^{jm}}^{2^{(j+1)m}}|f_{j}(t,y)|^2 t^\beta dtdy.
\end{align*}
 Call $I_{j}$ and $J_{j}$ the square roots of  the first and second integrals.   For $I_{j}$, we split
 the integral in $s$ defining $\mathcal{J}_{\alpha}A(t,y)$ as
 $$
 \sum_{k\ge 1}\  \int\limits _{2^{-k-1}t}^{2^{-k}t} \bigg(\frac{s}{t}\bigg)^\alpha (K(t,s)A(s,\cdot))(y)\, ds
 $$
 so that by Minkowski inequality $I_{j}\le \sum_{k\ge 1}I_{j,k}$ with
 $$
 I_{j,k}^2 =  \int\limits_{2^j<|y|<2^{j+1}}\int\limits_{0}^{2^{jm}} \Big| \int\limits _{2^{-k-1}t}^{2^{-k}t} \bigg(\frac{s}{t}\bigg)^\alpha(K(t,s)A(s,\cdot))(y)\, ds\Big|^2 t^\beta dtdy.
 $$
 Using Cauchy-Schwarz inequality in the $s$ integral
 and then the $L^q-L^2$ decay with $t\sim t-s$, we get
\begin{align*}
I_{j,k}^2 &\lesssim \int\limits_{0}^{2^{jm}}2^{-k} t  \int\limits_{2^{-k-1}t}^{2^{-k}t} \bigg(\frac{s}{t}\bigg)^{2\Re e\,\alpha}\frac{1}{t^{2+\frac{2n}{m}(\frac{1}{q}-\frac{1}{2})}}\Big(1+\frac{2^{jm}}{t}\Big)^{-2M}\,
\|A(s,\cdot)\|_{q}^2\,{ds}\ t^\beta dt
\\
& \lesssim  2^{-2jmM}  \int\limits_{0}^{2^{jm}} 2^{-k} t  \int\limits_{2^{-k-1}t}^{2^{-k}t} 2^{-2k\Re e\,\alpha}\frac{1}{t^{2+\frac{2n}{m}(\frac{1}{q}-\frac{1}{2})-2M}}
\|A(s,\cdot)\|_{2}^2\,{ds}\ t^\beta dt
\\
&
\lesssim   2^{-2jmM} 2^{k(-2 \Re e\, \alpha   + \beta-1)} \int\limits_{0}^{2^{jm-k}}  \|A(s,\cdot)\|_{2}^2\, s^\beta (2^ks)^{2M- \frac{2n}{m}(\frac{1}{q}-\frac{1}{2})}\, ds.
\end{align*}
Recall that  the support condition on $A$  forces $s\le 1$. Also $M> \frac{n}{2m}\ge \frac{n}{m}(\frac{1}{q}-\frac{1}{2})$. Using also the size requirement on $A$ we obtain
$$I_{j,k}^2 \lesssim
 2^{-2jmM} 2^{k(-2 \Re e\, \alpha   + \beta-1)}2^{\inf(k, jm)(2M- \frac{2n}{m}(\frac{1}{q}-\frac{1}{2}))} .
 $$
Hence, $\sum_{k\ge 1} I_{j,k}$ is controlled by $2^{-jm\inf (M, v(\alpha,q))}$ with $v(\alpha,q)= \Re e\,\alpha- \frac{\beta-1}{2}+ \frac{n}{m}(\frac{1}{q}-\frac{1}{2})$ if $M\ne v(\alpha,q)$ and by
$jm2^{-jmM}$ if $M= v(\alpha,q)$.

Next, for the second integral, we remark that the support of $A$ forces $s\le 1$ while $t\sim 2^{jm}\ge 2$. Hence
 \begin{align*}
\label{est:low}
J_{j}^2
&\leq
\int\limits_{|y|<2^{j}}\int\limits_{2^{jm}}^{2^{(j+1)m}}  \int\limits_{0}^{1} \bigg(\frac{s}{t}\bigg)^{2\Re e\,\alpha {-(\beta-1)}}
\Bigl|{t}(K(t,s){s^{\frac{\beta+1}{2}}}A(s,\cdot))(y)\Bigr|^2\,{\frac{ds}{s} \frac{dt}{t}}
 \\
 &
\lesssim \int\limits_{2^{jm}}^{2^{(j+1)m}}  \int\limits_{0}^{1} \bigg(\frac{s}{t}\bigg)^{2\Re e\,\alpha{-(\beta-1)}}\frac{{t^{2}}}{t^{\frac{2n}{m}(\frac{1}{q}-\frac{1}{2})+2}}\, \|{s^{\frac{\beta+1}{2}}}A(s,\cdot)\|_{q}^2\,{\frac{ds}{s} \frac{dt}{t}} \\
 &\lesssim
2^{-j(2(\Re e\,\alpha- \frac{\beta-1}{2})+ \frac{2n}{m}(\frac{1}{q}-\frac{1}{2}))m}= 2^{-2jmv(\alpha,q)}.
\end{align*}
We used H\"older's inequality, the size requirement on $A$, and also $s^{2\Re e\,\alpha- (\beta-1)}\le 1$.
In all
$$
\bigg(\ \int\limits_{|x|<2^{j+1}}\int\limits_{0}^{2^{(j+1)m}}|f_{j}(t,y)|^2 t^\beta dtdy\bigg)^{\frac{1}{2}} \lesssim (1+jm)2^{-jm\inf (M, v(\alpha,q))}.$$

We now start the discussion.
{Case (2) corresponds to} $v(0,q)>\frac{n}{2m}$. The exponent $\tilde p_{c}$ {is such that}  $v(0,q)=\frac{n}{2m}(\frac{2}{\tilde p_{c}}-{1})$.  By Lemma \ref{lem:atomic}, $\mathcal{J}_{0}$ extends to a bounded map on $ T^{p,2,m}(t^{\beta}dtdy)$ for any $p\le 1$ with
{$\frac{n}{2m}(\frac 2 p - 1) { \geq }\inf (M, v(0,q))$}, which means $1\ge p>\sup(p_{M},\tilde p_{c})$. By interpolation with the $p=2$ result, $\mathcal{J}_{0}$ extends to a bounded map on $ T^{p,2,m}(t^{\beta}dtdy)$ for $\sup(p_{M},\tilde p_{c})< p \le 2$.

{Case (1) corresponds to} $v(0,q)\le \frac{n}{2m}$. Let $\alpha_{1}>0$ be such that $v(\alpha_{1},q)=\frac{n}{2m}$. As in the preceding case, for any $\alpha$ with $ \Re e\, \alpha >\alpha_{1}$,  $\mathcal{J}_{\alpha}$ extends to a bounded map on $ T^{1,2,m}(t^{\beta}dtdy)$ and by checking the proof above, the bound does not depend on $ \Im m\, \alpha $. By the $p=2$ case, if $\alpha_{2}=\frac{\beta-1}{2}<0$, then  for any $\alpha$ with $ \Re e\, \alpha  >\alpha_{2}$,  $\mathcal{J}_{\alpha}$ extends to a bounded map on $ T^{2,2,m}(t^{\beta}dtdy)$ and  the bound does not depend on $ \Im m\, \alpha $. Hence, by Stein's interpolation theorem for analytic families extended to tent spaces (see \cite{htv} for its extension to the tent  spaces $T^{p,2}$ with $p\ge 1$),  $\mathcal{J}_{0}$ extends to a bounded map on $ T^{p,2,m}(t^{\beta}dtdy)$ for $p_{c}<p<2$ and $p_{c}$ is the exponent with $\frac{1}{p_{c}}=\frac{\theta}{1}+\frac{1-\theta}{2}$ when  $0= \theta \alpha_{1}+(1-\theta)\alpha_{
 2}$. A calculation yields the explicit formula of the statement.
\end{proof}

\begin{remark}
Note that the most restrictive conditions on $p$ come from the tail operator $T_{2}$, not the singular one $T_{1}$, which is contrary  to usual feeling for singular integral operators. This can be understood by noticing that this tail operator contains {the terms where $s$ is close to $0$,}
and some decay is required to control the  tent space norms {near this boundary}.
\end{remark}

We next give a result for operators in $\mT^-_{m,q,M}$ when $q\le 2$.

\begin{proposition}\label{prop:adjoint}
Let $\beta > -1,\,m \in \N^*$, $T\in \mT^-_{m,q,M}$ with $1\le q \le 2$ and $M> \frac{n}{2m}.$ Let $p_{M}<1$ be such that $M=\frac{n}{2m}(\frac{2}{p_{M}-1})$.
Then $T$ extends to a bounded operator on $T^{p,2,m}(t^\beta dt dx)$ for $p_{M}<p<2$.
\end{proposition}

\begin{proof} By interpolation, it suffices to treat the case  $p_{M}<p\le 1$.  Take such a $p$.
Let $A$ be a $T^{p,2,m}(t^\beta dt dy)$ atom, i.e. a function
supported in some $(0,r^m] \times B(x_0,r),$ and satisfying
\[ \iint\limits_{\R^{n+1}_+} |A(s,x)|^2 s^\beta ds dx \leq r^{-n(\frac{2}{p}-1)}. \]
For $j \in \N,$ let $B_j = (0,(2^{j}r)^{ m}] \times B(x_0,2^jr) \subset \R^{n+1}_+$ and $C_j = B_j \backslash B_{j-1}$
(with $B_{-1} = \emptyset$).
For $k,j \in \N,$ and $(k,j)\ne (0,0)$ we let
\[ T_{k,j} A(t,y) = 1_{C_j}(t,y)\int\limits_{2^{k} t}^{2^{k+1} t} (K(t,s) A(s,\cdot))(y) ds \]
and \[ (T_{0,0} A)(t,y)= 1_{B_0}(t,y)(T_{1}A)(t,y)\]
where $T_{1}$ is the singular part of $T$.

We claim that, for a sequence $\lambda_{k,j} > 0,$ which is independent of $A$ and satisfies $\sum_{k,j = 0}^\infty \lambda_{k,j} < \infty,$ we have
\begin{equation*}\label{equ:proof adjoint}
\iint\limits_{B_j} |T_{k,j} A(t,y)|^2 t^\beta dt dy \leq (2^jr)^{-n(\frac{2}{p}-1)} \lambda_{k,j}^2,
\end{equation*}
so
$\lambda_{k,j}^{-1} T_{k,j} A$ is a $T^{p,2,m}(t^\beta dt dy)$ atom.
Note that $\sum_{k\ge 1, j\ge 0} T_{k,j}A= T_{2}A$.
Using Lemma \ref{lem:rep} a posteriori, we have $T_{1}A=\sum_{j\ge 0} T_{0,j}A$. Hence $\sum_{k\ge 0, j\ge 0} T_{k,j}A= TA$ and thus   $\| T A \|_{T^{p,2,m}(t^\beta dt dy)} \lesssim \sum_{k,j =0}^\infty \lambda_{k,j}.$
By Lemma \ref{lem:atomic}, we are then able to conclude the proof.
It remains to prove the claim.

{The proof is scale and translation invariant so} we assume that $x_0 = 0$ and $r = 1.$
For $j \geq 1,$ we have
\begin{align*}
\iint\limits_{B_j}
& \left| T_{k,j} A(t,y) \right|^2 t^\beta dt dy
 \leq \iint\limits_{C_j}(2^k t)^{2 \epsilon}  \int\limits_{2^k t}^{2^{k+1} t} (s-t)^{1 - 2\epsilon} \left| (K(t,s)A(s,\cdot) )(y)\right|^2 ds\, t^\beta dtdy.
\end{align*}
Here we have used the Cauchy-Schwarz inequality {as in the proof of Lemma \ref{lem:m1} and  the parameter $\epsilon > { 0}$ will be determined later.}
Write   $C_j = (0,2^{(j-1)m}] \times \left[ B(0,2^{j}) \backslash B(0,2^{j-1}) \right] \cup [2^{(j-1)m},2^{jm}] \times B(0,2^j) =: C_j^{(1)} \cup C_j^{(2)}.$
If $(t,y) \in C_j^{(2)},$ then $t \geq 2^{(j-1)m} \geq 1,$ and if $s > 2^{k}t \geq 1,$ then $A(s,\cdot) = 0.$
Thus, we can replace $C_j$ by $C_j^{(1)}$ in the above multiple integral and impose $t\le 1$.
Then we can apply the $L^q-L^2$ decay  with $F = B(0,2^j)\backslash B(0,2^{j-1})$ and $E = B(0,1)$  to continue estimating as follows
\begin{align*}
& \leq \int\limits_0^{1}  (2^k t)^{2 \epsilon}\int\limits_{2^k t}^{2^{k+1}t}  \frac{1}{(s-t)^{1+2\varepsilon+\frac{2n}{m}(\frac{1}{q}-\frac{1}{2})}} ( 1 + \frac{2^{jm}}{s-t})^{-2M} \left\| A(s,\cdot)\right\|^2_{q} ds \,  t^\beta dt  \\
& \cong  \int\limits_0^1 (2^k t)^{2 \epsilon} \int\limits_{2^k t}^{2^{k+1}t} \frac{1}{(s-t)^{1+2\varepsilon+ \frac{2n}{m}(\frac{1}{q}-\frac{1}{2})}}(1 + \frac{2^{jm}}{s-t})^{-2M} \left\|{ A(s,\cdot)}\right\|^2_{2} ds  \, t^\beta dt  \\
& = 2^{2k \epsilon}  \int\limits_0^{2^{k+1}}  \int\limits_{2^{-k-1}s}^{2^{-k}s}  \frac{t^{\beta+2\varepsilon}}{(s-t)^{\frac{2n}{m}(\frac{1}{q}-\frac{1}{2})+1+2\varepsilon}} ( 1 + \frac{2^{jm}}{s-t} )^{-2 M} dt \, \|A(s,\cdot)\|^2_{2}\, ds  \\
& \cong 2^{2k \epsilon}  \int\limits_0^1 ( 2^{-k}s)^{\beta + 2 \epsilon} \int\limits_{2^{-k-1}s}^{2^{-k}s} \frac{1}{(s-t)^{\frac{2n}{m}(\frac{1}{q}-\frac{1}{2})+1+2\varepsilon}}(\frac{2^{jm}}{s-t})^{-2M} dt\,  \|A(s,\cdot)\|^2_{2}\, ds .
\end{align*}
We take $\epsilon \in ({0},M- \frac{n}{m}(\frac{1}{q}-\frac{1}{2}))$ so that the integral with respect to  $t$ converges.
Indeed, $M>\frac{n}{2m} \ge  \frac{n}{m}(\frac{1}{q}-\frac{1}{2})$ and the calculation continues as follows:
\begin{align*}
& \cong 2^{2k\epsilon} 2^{-k(\beta + 2 \epsilon)}2^{-2Mmj} 2^{-k}  \int\limits_0^1 s^{2M-\frac{2n}{m}(\frac{1}{q}-\frac{1}{2}) - 2 \epsilon} \|A(s,\cdot)\|^2_{2}\,  s^{\beta + 2 \epsilon} ds  \\
& \lesssim   2^{-k(\beta + 1)} 2^{-2Mmj}\int\limits_0^1\|A(s,\cdot)\|^2_{2} \, s^\beta ds \\
& \leq 2^{-jn(\frac{2}{p}-1)} \lambda_{k,j}^2
\end{align*}
with $\lambda_{k,j} \cong 2^{(\frac n2(\frac{2}{p}-1) - Mm)j} 2^{-\frac{k}{2}(\beta + 1)}$, and we used $M>\frac{n}{2m}\ge  \frac{n}{m}(\frac{1}{q}-\frac{1}{2})$.

If $j=0$ and $k \geq 1,$ we do not use the decay but rather the fact that $(t-s)K(t,s)$ is uniformly bounded on $L^2(\R^n).$
Then we can repeat the above calculation literally taking $q=2$ and $M=0$.

If $k=0$ and $ j=0,$ using the boundedness of $T_{1}$ since $\beta>-1$,
\[ \int\limits_{B(0,2)} \int\limits_0^{2^m} \left| (T_{0,0}A)(t,y) \right|^2 t^\beta dt dy \leq C \int\limits_{B(0,1)} \int\limits_0^1 |A(s,x)|^2 s^\beta ds dx. \]
We conclude that $\lambda_{k,j} \cong 2^{(\frac n2(\frac{2}{p}-1) - Mm)j} 2^{-\frac{k}{2}(\beta + 1)}$ is summable for $\beta > -1$ and $M > \frac{n}{2m}(\frac{2}{p}-1).$
\end{proof}

\section{Role of $L^2-L^q$ decay}
{When $q\ge 2$,   $L^2-L^q$ decay can be used to quantify $T^{p,2}$ results for $p$ above 2.} Clearly the adjoint class to $\mT^\pm_{m,q,M}$ is $\mT^\mp_{{m,q',M}}$ with respect to the
inner product $$\langle f , g \rangle = \int\limits_{\R^n} \int\limits_0^\infty f(t,y) \overline{g(t,y)} dt dy.$$
It is easy to deduce from \cite[Section 5]{cms} that
for $p \in (1,\infty),m \in \N^*$ and $\beta \in \R,$ we have $\left[T^{p,2,m}(t^\beta dt dy)\right]' = T^{p',2,m}(t^{-\beta} dt dy),$
with duality given by $\langle f, g\rangle,$ i.e.
\[\|f\|_{T^{p,2,m}(t^\beta dt dy)} \sim \sup_{\|g\|_{T^{p',2,m}(t^{-\beta} dt dy)} \leq 1}|\langle f , g \rangle|. \]

Thus, we obtain results for $2<p<\infty$ by dualizing Theorem \ref{thm:lql2abs} and Proposition \ref{prop:adjoint} in the classes $\mT^\pm_{m,q,M}$ with $2\le q\le \infty$ and $M>\frac n {2m}$. In addition, the results for $p=\infty$ also hold.

\begin{theorem}\label{thm:lql2absdual}   Let $T\in \mT^-_{m,q,M}$ with  $2\le q\le \infty$ and  $M>\frac{n}{2m}$.  Let $\beta>-1$.
 \begin{enumerate}
  \item  If $q\le \frac{2n}{m(1-\beta)}$ or equivalently  $\frac{n}{2m} \ge -\frac{\beta-1}{2} + \frac{n}{m}(\frac{1}{2}-\frac{1}{q})$
   then
 ${T}$  extends to a bounded operator on
 $T^{p,2,m}(t^{\beta}dtdy)$ when $2\le p<p_{c}'$, where
 $$p_{c}=\frac{2\left(\frac{n}{2m}-\frac{n}{m}\left(\frac{1}{2}-\frac{1}{q}\right)\right)}{\frac{n}{2m}-\frac{n}{m}\left(\frac{1}{2}-\frac{1}{q}\right)+ \frac{1-\beta}{2}}= \frac{4n}{2n + m(1-\beta)q}\ .
$$
 \item
 If $q> \frac{2n}{m(1-\beta)}$ or equivalently   $  -\frac{\beta-1}{2} + \frac{n}{m}(\frac{1}{2}-\frac{1}{q})> \frac{n}{2m}$
 then
 $T$  extends to a bounded operator on
 $T^{p,2,m}(t^{\beta}dtdy)$ when $2\le p\le \infty$.
  \end{enumerate}
\end{theorem}

\begin{proposition}\label{prop:adjointdual}
Let $T\in \mT^+_{m,q,M}$ with $2\le q \le \infty$ and    $M > \frac{n}{2m}$. For  all $\beta <1$,
 $T$ extends to a bounded operator on $T^{p,2,m}(t^\beta dt dy)$ for $2\le p\le \infty$.
\end{proposition}

It is enough to prove the result for $p=\infty$. The extension is done by taking $f\in T^{\infty,2,m}(t^\beta dt dy)$, truncating $f$ on $(0,k^m)\times B(0,k)$ and letting $k$ go to infinity.

\begin{proof}[Proof of Proposition \ref{prop:adjointdual}] This is very similar to \cite{amp}.
Pick a point $x_{0}\in \R^n$ and $r>0$. Let the sets $B_{j}$ and $C_{j}$ be defined as in the proof of Proposition \ref{prop:adjoint}.  Set
$$
I^2=\int \limits _{B(x_{0},r)}
\int \limits _{0} ^{r^m} |(Tf)(t,y)|^{2} t^{\beta} dtdy.
$$
We want to show that $I^{2} \lesssim r^{{n}} \|f\|_{T^{\infty,2,m}(t^{\beta}dtdy)}^2$.
We set
$$
I_{j}^2=
\int \limits _{B(x_{0},r)} \int \limits _{0} ^{r^m} |(Tf_{j})(t,y)|^{2} t^{\beta} dtdy
$$
where $f_{j}(s,x)=f(s,x) 1_{C_{j}}(s,x)1_{(0,r^m]}(s)$ for $j\ge 0$.  Thus by
Minkowski inequality, $I \le \sum I_{j}$.
{Since the proofs are scale and translation invariant}, we assume $x_{0}=0$ and $r=1$ for simplicity.
For $I_{0}$ we use again Theorem~\ref{AAabs}
which implies
$$
I_{0}^2\lesssim  \int \limits _{B(0,2)} \int \limits _{0} ^{2^m} |f(s,x)|^{2} s^{\beta} dsdx
\lesssim  \|f\|_{T^{\infty,2,m}(t^{\beta}dtdy)}^2.
$$
Next, for $j \neq 0$, we proceed as in the proof of Proposition \ref{prop:adjoint}  by representing $Tf_{j}(t,y)$ through a kernel (which is justified by the calculation below and Lemma \ref{lem:rep} for the singular part)  but using this time $L^2-L^q$ decay {(after using H\"older inequality for the integral with respect to $y$ on $B(0,2)$)} to obtain
\begin{eqnarray*}
I_{j}^{2}
&\lesssim& \sum \limits _{k=1} ^{\infty} \int \limits _{0} ^{1} \int \limits _{2^{-k-1}t} ^{2^{-k}t}
  \frac{2^{-k}t}{|t-s|^{\frac{2n}{m}(\frac{1}{2}-\frac{1}{q})+2}}\Bigl(1+\frac{2^{jm}}{t-s}\Bigr)^{-2M} \|f_{j}(s,.)\|_{{2}} ^{2} \,ds\,t^{\beta}dt
\\[4pt]
&&+ \int \limits _{0} ^{1} \int \limits _{\frac{t}{2}} ^{t} \frac{t^{\beta+2\varepsilon}}{|t-s|^{\frac{2n}{m}(\frac{1}{2}-\frac{1}{q})+1+2\varepsilon}} \Bigl(1+\frac{2^{jm}}{t-s}\Bigr)^{-2M}
\|f_{j}(s,.)\|_{{2}} ^{2}\, \,ds\,dt.
\end{eqnarray*}
Exchanging the order of integration, and using the fact that $t \sim t-s$ in the first part and
that $t \sim s$ in the second, and $M>\frac{n}{m}(\frac{1}{2}-\frac{1}{q})+\varepsilon$ for small enough $\varepsilon$,  and  $\beta<1$,  we have the following.
\begin{eqnarray*}
I_{j}^{2} &\lesssim&
\sum \limits _{k=1} ^{\infty} 2^{-k}2^{-2jmM}
\int \limits _{0} ^{2^{-k}} \int \limits _{2^{k}s} ^{2^{k+1}s} t^{\beta-1+2M-\frac{2n}{m}(\frac{1}{2}-\frac{1}{q})}
\|f_{j}(s,.)\|_{{2}} ^{2} \, dtds
\\[4pt]
&& + \int \limits _{0} ^{1} \int \limits _{s} ^{2s} \frac{t^{\beta+2\varepsilon}}{|t-s|^{\frac{2n}{m}(\frac{1}{2}-\frac{1}{q})+1+2\varepsilon}}
\Bigl(1+\frac{2^{jm}}{t-s}\Bigr)^{-2M}\|f_{j}(s,.)\|_{{2}} ^{2}\, s^{\beta}dtds
\\[4pt]
&\lesssim&
\sum \limits _{k=1} ^{\infty} 2^{-k}2^{-2jmM}
\int \limits _{0} ^{2^{-k}} (2^{k}s)^{\beta} \|f_{j}(s,.)\|_{{2}} ^{2}\, ds
+ 2^{-2jmM} \int \limits _{0} ^{1}
\|f_{j}(s,.)\|_{L^{2}} ^{2} s^{\beta} ds
\\[4pt]
&\lesssim& 2^{-2jmM} \int \limits _{0} ^{2^{jm}}  \|f_{j}(s,.)\|_{{2}} ^{2}\, s^{\beta}ds.
\end{eqnarray*}
We thus have
$$
I_{j}^{2} \lesssim 2^{-2jmM}2^{jn}\|f\|^{2}_{T^{\infty,2,m}(t^{\beta} dtdy)},
$$
and the condition $M>\frac{n}{2m}$ allows us to sum these estimates.
\end{proof}

\begin{proof}[Proof of Theorem  \ref{thm:lql2absdual}] The proof is almost entirely similar to the above one.  Set $I_{j}$ as in the proof of Proposition \ref{prop:adjoint}. $I_{0}$ is estimated as before. When $j\ge 1$, the inner term in $ I_{j}$ can be expressed using the kernel representation  from $t$ to $+\infty$, which is split into $I_{j,k}$ on the dyadic  intervals $(2^kt, 2^{k+1}t)$ for $k\in
\N$, using Minkowski inequality. The $k=0$ term is estimated as was the term corresponding to $(t/2,t)$. For $k\ge 1$, the $k$th term is controlled by
$$
 \int \limits _{0} ^{1} \int \limits _{2^{k}t} ^{2^{k+1}t}
  \frac{2^{k}t}{|t-s|^{\frac{2n}{m}(\frac{1}{2}-\frac{1}{q})+2}}\Bigl(1+\frac{2^{jm}}{t-s}\Bigr)^{-2M} \|f_{j}(s,.)\|_{{2}} ^{2}\, ds\,t^{\beta}dt.
  $$
  Exchanging order, we obtain the bound
  $$
 2^{-2jmM} 2^{k(1-\beta+ 2M- \frac{2n}{m}(\frac{1}{2}-\frac{1}{q}))}  \int \limits _{0} ^{2^k}
\|f_{j}(s,.)\|_{{2}} ^{2} \,s^{\beta} ds.
 $$
 Note that the support of $f_{j}$ forces $s\le 2^{(j+1)m}$ in the integral, which is bounded by $C 2^{jn}$.
  The series for $I_{j,k}$ is summable in $k$ under the condition in the statement and summable in
  $j$ if  $M>\frac{n}{2m}$.
\end{proof}

\section{Maximal regularity operators}\label{sec:max}
Let us come back to our original motivation {which is to bound  maximal regularity operators on tent spaces}.

\begin{definition} Let $1\le q\le r\le \infty$.
A family of bounded linear operators $(T_{t})_{t > 0} \subset B(L^{2}(\R^{n}))$ is said to
satisfy $L^q-L^r$ off-diagonal estimates of order $M$, with homogeneity $m$, if, for all Borel sets
$E,F \subset \R^{n}$, all $t>0$, and all $f \in L^{2}(\R^{n})\cap L^q(\R^n)$:
$$
\|1_{F}T_{t}1_{E}f\|_{r} \lesssim t^{-\frac{n}{m}(\frac{1}{q}- \frac{1}{r})}\,  \Big(1+\frac{dist(E,F)^{m}}{t}\Big)^{-M}\|1_{E}f\|_{q}.
$$
\end{definition}

{With this definition we have the following simple fact.}

\begin{proposition}
Let $1\le q\le 2$ (resp. $2\le q\le \infty$) and assume that $(tLe^{-tL})_{t \geq 0}$ satisfies $L^q-L^2$ (resp. $L^2-L^q$) off-diagonal estimates (of order
$M$),
 with homogeneity $m$. Then
 $\mathcal{M}_{L} \in \mT^+_{m,q,M}$ and $\mathcal{M}_{L^*}^-\in \mT^-_{m,q',M}$.  \end{proposition}

 {Indeed,  the operator-valued kernel $Le^{-|t-s|L}$ has $L^q-L^2$ (resp. $L^2-L^q$) decay (of order
$M$),
 with homogeneity $m$ so that it suffices to apply Definition \ref{def:SIOq}.}

To illustrate  our results so far, let us prove Proposition \ref{cor:Laplace}.

{
\begin{proof}[Proof of Proposition \ref{cor:Laplace}] Let $L=-\Delta+V$ or $-{\rm div} \,A\nabla$ with real coefficients.
{
Then, the kernel of the semigroup $(e^{-tL})_{t\geq 0}$ satisfies  pointwise Gaussian estimate {(see e.g. \cite[Theorem 6.10]{o})}, hence $L^1-L^2$ and $L^2-L^\infty$ off-diagonal estimates with homogeneity   $m=2$ of order $\infty$. {Therefore we have that} $\mathcal{M}_{L} \in \mT^+_{2,1,\infty}\cap \mT^+_{2,\infty,\infty}$. We now   apply  the second case of Theorem \ref{thm:lql2abs} and Proposition \ref{prop:adjointdual} with  $\beta=0$ to conclude that  $T^{p,2,m}(dt dy)$ boundedness of  $\mathcal{M}_{L}$ holds for $\infty\ge p>\tilde p_{c} =\frac{n}{n+1}$.

 Using the subordination formula,  the Poisson semigroup associated with $\sqrt L$ satisfies $L^1-L^2$ and $L^2-L^\infty$ off-diagonal estimates with homogeneity   $m=1$ and order $\frac n 2 +1$. Thus  $\mathcal{M}_{\sqrt{L}} \in \mT^+_{1,1,\frac n 2 +1 }\cap \mT^+_{1,\infty,\frac n 2 +1}$.  From $M=\frac n 2 +1$ and $m=1$, we have $p_{M}= \frac{n}{n+1}$.
As $\beta=-1$, $m=1$ and $q=1$,   $\frac n 2 < -\frac{\beta-1}{2} + \frac{n}{m}\left(\frac{1}{q}-\frac{1}{2}\right) =1+ \frac n 2$ and we are in the second case of Theorem \ref{thm:lql2abs}. Applying this result  and Proposition \ref{prop:adjointdual}, we conclude that $T^{p,2,m}(t ^{-1}dt dy)$ boundedness of $\mathcal{M}_{\sqrt{L}}$ holds for $\infty\ge p>\sup(p_{M},\tilde p_{c}) =\frac{n}{n+1}$. }
\end{proof}

As explained in the introduction,  applications of our results require  $M$ to be sufficiently large, namely $M>\frac{n}{2m}$, whatever the value of $q$}. Of course, with exponential decay, this is not a  problem.
{Semigroups generated by elliptic operators of even  order $m\ge 2$ have, in general, such an exponential off-diagonal decay}.
However, in the case of Poisson type semigroups, small  polynomial decay is to be expected.
This application suggests  {that
the lower bound on $M$ should be kept as low as possible}. Looking at the proof of Lemma \ref{lem:m1}, {there seems to be unavoidable restrictions}
if we are only given $M$  without further information. However, the decay of the semigroup is usually computed from the decay of the resolvent and integration on a contour. This is the point of view we shall take.

We consider the  following conditions on resolvent estimates for fixed $1\le q\le r\le \infty$.

1)  There exists a bisectorial operator $\widetilde{L}$ of angle $\omega \in [0,\frac{\pi}{2})$ having a bounded $H^\infty$ functional calculus on $L^2(\R^n)$ such that $L = |\widetilde{L}| (=\sqrt{\widetilde L^2}=\sqrt {L^2}),$ and  for any $K\in \N$ and $\omega<\nu<\pi/2$,
\begin{equation}\label{equ:OD resolvent |L|}\tag{H1}
\|1_F(1-z\widetilde{L})^{-1}1_E f\|_r \leq c(K,\nu) |z|^{-\frac{n}{m}(\frac{1}{q}- \frac{1}{r})} \Big( 1 + \frac{dist(E,F)^m}{|z|} \Big)^{-K} \|1_E f\|_q .
\end{equation}
for all $f\in L^2(\R^n)\cap L^q(\R^n)$,  $E,F$ Borel subsets of $\R^n$, $z=e^{\pm i \theta}t$, $t>0$ and  $|\theta- \frac{\pi}{2}| <  \frac{\pi}{2} - \nu$.

2) The operator $L^2$ is  sectorial in $L^2(\R^n)$ of angle $2 \omega < \pi$  and for any $K\in \N$ and $\omega<\nu<\pi/2$,
\begin{equation}\label{equ:OD resolvent L2}\tag{H2}
\|1_F(1-z L^2)^{-1}1_E f\|_r \leq c(K,\nu) |z|^{-\frac{n}{2m}(\frac{1}{q}- \frac{1}{r})} \Big( 1 + \frac{dist(E,F)^{2m}}{|z|} \Big)^{-K} \|1_E f\|_q
\end{equation}
for all $f\in L^2(\R^n)\cap L^q(\R^n)$,  $E,F$ Borel subsets of $\R^n$, $z=e^{\pm i \theta}t$, $t>0$ and  $2\nu < \theta\le \pi$.

Operators of Dirac type satisfying \eqref{equ:OD resolvent |L|} with $m=1$ appear in \cite[Proposition 5.2]{akm}.
{See also \cite{amr} and \cite{hnp}.}

(H1) and (H2) are closely related and, in fact,    (H1) implies (H2). Indeed, it follows from the
resolvent formula
$$
(1-z^2 \widetilde L^2)^{-1} = \frac{1}{2} (1-z\widetilde{L})^{-1} + \frac{1}{2} (1+z\widetilde{L})^{-1}
$$
for $z$ as in (H1).  Remark that, in (H2), $2w$ may be greater than or equal to $\pi/2$, in which case $-L^2$ may not generate a semigroup.

\begin{proposition}\label{prop:qr} Let $L$ be a sectorial operator of angle $\omega<\pi/2$ with an $H^\infty$ functional calculus on $L^2(\R^n)$.
Assume that \eqref{equ:OD resolvent |L|} or \eqref{equ:OD resolvent L2} is satisfied and fix $\omega<\nu<\pi/2$.
Then for any $0< \epsilon < R < \infty$  and  any $\alpha \in \C$ with $ \Re e\, \alpha  \in [\epsilon , R],$ $\|1_F (tL)^\alpha e^{-tL} 1_Ef \|_r$ has bound
\begin{equation*}\label{equ:OD M alpha}
  c(\epsilon,R,q,r,\nu) e^{\nu| \Im m\, \alpha \Bk |} \cdot t^{-\frac{n}{m}(\frac{1}{q}- \frac{1}{r})} \Big( 1 + \frac{dist(E,F)^m}{t} \Big)^{- \Re e\, \alpha  - \frac{n}{m}(\frac{1}{q}- \frac{1}{r})} \|1_E f\|_q.
\end{equation*}
\end{proposition}

A result in this spirit is in \cite{hm} for $q=r=2$.

\begin{proof} It is enough to assume (H2).
In this case, fix $\omega<\nu'<\theta<\nu$, and  let \[\phi_t(\lambda) = (t\lambda^{\frac12})^\alpha e^{-t\lambda^{\frac12}}\]
which is holomorphic and bounded for $|\arg \lambda\, | < \pi - 2\nu'.$
The Cauchy integral formula for sectorial operators implies that
\[ (tL)^\alpha e^{-tL} = \frac{1}{2\pi i }\int\limits_\Gamma \phi_t(\lambda)(1- \lambda^{-1}L^2)^{-1} \frac{d\lambda}{\lambda}\]
holds with $\Gamma$ the oriented contour $\{|s|e^{i \text{sign}(s)2\theta}:\:s \in \R\}.$
We write $c_\theta = \Re e\, (e^{i\theta})=\cos \theta > 0.$
Fix $f$ with $\|1_E f\|_q=1$. In the following, we write $a = \Re e\, \alpha  ,\, d = dist(E,F)$, $\gamma=\frac{n}{m}(\frac{1}{q}- \frac{1}{r})$.  Then \eqref{equ:OD resolvent L2}  gives us
\begin{align*}
\| 1_F (tL)^\alpha e^{-tL} 1_E f\|_r & \lesssim c(K) \int\limits_0^\infty t^a s^{{\frac{a}{2}}} e^{{\theta}  \Im m\, \alpha} e^{-c_\theta ts^{\frac12}}  s^{\frac{\gamma}{2}} ( 1 + d^{2m}s)^{-K} \frac{ds}{s} \\
& \cong c(K,\nu)e^{\nu| \Im m\, \alpha  |}\int\limits_0^\infty t^a s^a  s^{\gamma} e^{-c_\theta t s}(1+d^{2m}s^2)^{-K} \frac{ds}{s} \\
& \lesssim  c(K,\nu)e^{\nu|\Im m\, \alpha  |} t^{-\gamma}\int\limits_0^\infty  s^a  s^{\gamma} e^{-c_\theta s}\Big(1+\frac{d^{2m}s^2}{t^2}\Big)^{-K} \frac{ds}{s} \\
& \lesssim  c(\epsilon, K, q,r,\nu)e^{\nu| \Im m\, \alpha  |} t^{-\gamma} \Big(1 + \frac{d^m}{t}\Big)^{-a-\gamma}.
\end{align*}
provided   $2K > R+\gamma$. {We used the fact that $1 \le  2(1 + x)^{-1} $ when $x \leq 1$, and
$ x^{-1} \le 2 (1 +x)^{-1}$ when $x \geq 1$.} The parameter $\epsilon>0$ is only needed when $q=r$.
\end{proof}

It is clear that similar results hold for fractional powers of sectorial operators.  We shall not get into this here. Note also that an exponential decay in the resolvent estimates would not yield a better conclusion in general.

\begin{definition} Let $L$ be a sectorial operator of type $\omega<\pi/2$ and having a bounded holomorphic functional calculus on a Hilbert space $H$.
For $ \Re e\, \alpha   > 0$, we define the operator $\mM_\alpha$ acting on $L^2(\R_{+}, dt;  D(L^\alpha))$ by
\[ \mM_\alpha f (t) = \int\limits_0^t (t-s)^{\alpha - 1} L^\alpha e^{-(t-s)L}f(s)\,  ds.\]
\end{definition}

Clearly $\mM_{1}=\mM_{L}$.

\begin{proposition}\label{prop:malpha}
{Let $\alpha \in \{z \in \C \;;\; a \leq \Re e\, z \leq b\}$ for some $a,b \in \R_{+}$.}
Then $\mM_{\alpha}$ extends  boundedly to $L^2(\R_{+}, dt; H)$, with a bound not exceeding $ce^{\nu | \Im m\, \alpha  |}$ for any $\omega<\nu<\pi/2$, {and some
constant $c$  dependent on $a,b$.}
\end{proposition}

\begin{proof} {Using operational calculus as in \cite{aa}, which is possible  since $L$ has bounded holomorphic functional calculus on $H$, } it is enough to prove the same thing for $L=zI$ on $L^2(\R_{+}, dt; \C)$ for $|\arg z \, | <\nu$. In this case, we use Schur's lemma for the complex-valued kernel  $(t-s)^{\alpha - 1} z^\alpha e^{-(t-s)z} 1_{s<t}$.  For $w=\Re e\, z$, $|z|\le \frac{w}{\cos\nu}$, hence
\begin{align*}
\int\limits_{0}^t |(t-s)^{\alpha - 1} z^\alpha e^{-(t-s)z}|\, ds  &\le \frac{e^{\nu | \Im m\, \alpha  |}}{(\cos\nu)^{ \Re e\, \alpha  }}  \int\limits_{0}^t |(t-s)^{ \Re e\, \alpha   - 1} w^{ \Re e\, \alpha  } e^{-(t-s)w}|\, ds \\
&
\le \frac{\Gamma( \Re e\, \alpha  ) e^{\nu | \Im m\, \alpha |}}{(\cos\nu)^{ \Re e\, \alpha }}
\end{align*}
and
\begin{align*}
\int\limits_{s}^\infty |(t-s)^{\alpha - 1} z^\alpha e^{-(t-s)z}|\, dt & \le \frac{e^{\nu | \Im m\, \alpha  |}}{(\cos\nu)^{ \Re e\, \alpha  }}  \int\limits_{s}^\infty |(t-s)^{\Re e\,\alpha - 1} w^{ \Re e\, \alpha  } e^{-(t-s)w}|\, dt \\ &
\le \frac{\Gamma( \Re e\, \alpha ) e^{\nu | \Im m\, \alpha  |}}{(\cos\nu)^{ \Re e\, \alpha }},
\end{align*}
with $\Gamma$ being the Euler  Gamma function.
\end{proof}

\begin{corollary}\label{cor:ML} Let $H=L^2(\R^n)$. If  $1\le q \le \infty$ and (H2) holds for $(q,2) $ if $q\le 2$ or $(2,q)$ if $q\ge 2$ then $\mM_\alpha \in \mT^+_{m,q, M_{q}}$ with $M_{q}= \Re e\, \alpha   + \frac{n}{m}|\frac{1}{q}- \frac{1}{2}|$.

\end{corollary}

{ Observe that the order of decay becomes a function of $q$, hence the notation $M_{q}$. $M_{q}$ increases as $q$ moves away from $2$: this is the interesting point for us.
As mentioned in the introduction, $M_{2}=1$ is best possible for the Poisson semigroup of $-\Delta$,  so it seems one cannot improve this conclusion. }

\begin{proof}
The fact that $\mM_\alpha \in \mT^+$  is contained in  Proposition \ref{prop:malpha}. The decay of the kernel $(t-s)^{\alpha - 1} L^\alpha e^{-(t-s)L}$ with $s<t$ is clear from Proposition \ref{prop:qr}.
\end{proof}

\begin{corollary}\label{cor:MLtp2} Let $H=L^2(\R^n)$.

A] Assume (H2) holds for $(q,2) $ with $q\le 2$.
 Then $\mM_L$ extends to a bounded operator on $T^{p,2,m}(t^{\beta}dt dy)$ for $p_{L}<p<2$ with $p_{L}$ calculated as follows:
 \begin{enumerate}
  \item If $\frac{n}{mq'}<1$ and $\beta\le -1$, $p_{L}=p_{M_{q}}$.
  \item If $\frac{n}{mq'}<1$ and $-1<\beta<1$, $p_{L}=\inf (\tilde p_{c}, p_{c})$.
 \item  If $\frac{n}{mq'}\ge 1$ then $\frac{1}{p_{L}}- \frac{1}{2}= \frac{mq'}{n}(\frac{1}{\inf(p_{c},1)}-\frac{1}{2})$.
\end{enumerate}

B] Assume  (H2) holds for $(2,q) $ with $q\ge 2$.
 Then $\mM_L$ extends to a bounded operator on $T^{p,2,m}(t^{\beta}dt dy)$ for $2<p<p_{L}$ with $p_{L}= \frac{2n}{n-mq}$ if $mq\le n$ and for $2<p\le \infty$ if $mq>n$.
 \end{corollary}

 Note that the result for $p\ge 2$ does not depend on $\beta$. The exponents $p_{M_{q}}$, $\tilde p_{c}$, $p_{c}$ are explicitely defined in Theorem \ref{thm:lql2abs}.
 The last two depend on $\beta$.

 \begin{proof}
 A] The condition $\frac{n}{mq'}<1$ is equivalent to $M_{q}=1+ \frac{n}{m}(\frac{1}{q}-\frac 1  2) >\frac{n}{2m}$. Cases (1) and (2) thus follow from
 Theorem \ref{thm:lql2abs}. In the third case, Theorem \ref{thm:lql2abs} does not apply to $\mM_L$ but to
$\mM_{\alpha}$ for any $\alpha$ with $ \Re e\, \alpha  >\alpha_{1}$ and  $\alpha_{1}+ \frac{n}{m}(\frac{1}{q}-\frac 1  2) =\frac{n}{2m}$ which implies that $\mM_{\alpha}$ is bounded for $\inf (p_{c},1)<p<2$. At the same time, $\mM_{\alpha}$ is bounded for $p=2$ when $\Re e\, \alpha  >0$. The third case follows by complex interpolation {for the analytic family $\mM_{\alpha}$} (since the growth in $\ \Im m\, \alpha $ is admissible) in tent spaces.

 B] The condition $mq>n$ means $M_{q}=1+ \frac{n}{m}(\frac{1}{2}-\frac 1  q) >\frac{n}{2m}$.  So we apply Proposition \ref{prop:adjointdual} to $\mM_{L}$. If $mq\le n$, then we apply  this result   not to  $\mM_{L}$ but to $\mM_{\alpha}$ for $ \Re e\, \alpha  >\alpha_{1}$ and $\alpha_{1}+ \frac{n}{m}(\frac{1}{2}-\frac 1  q) =\frac{n}{2m}$ and the $p=2$ result for $\Re e\,\alpha>0$ and conclude by interpolation { for analytic families} again.
\end{proof}

{
\begin{proof}
[Proof of Propositions \ref{prop:divagrad} and \ref{prop:sqrtdivagrad}] Write $L=-{\rm div} \,A\nabla$.  We have that $(e^{-tL})_{t\ge 0}$ satisfies pointwise Gaussian estimates  if $n=1,2$.  Hence the conclusion  of the first part of Proposition \ref{cor:Laplace} applies.
For $n\ge 3$,  let $1\le p_{-}(L) < p_{+}(L)\le \infty$ be the numbers introduced  in \cite{memoir} such that for  $p_{-}(L)<q\le r <p_{+}(L)$, $(e^{-tL})_{t\ge 0}$ satisfies $L^q-L^r$ off-diagonal estimates with homogeneity $m=2$. As the decay is Gaussian, the order is $\infty$.  Moreover, $p_{-}(L) < \frac{2n}{n+2}$, $p_{+}(L)>\frac{2n}{n-2}$ and,
by \cite{hmm},  this is sharp for this class of complex operators.  Taking  $q<\frac{2n}{n+2}$, we use the {second} item  in Corollary \ref{cor:MLtp2}, A]  when $n=3,4$ and the third one when $n\ge 5$ to get the lower bound on $p$. For the upper bound $p=\infty$ included, we use B].

Now for the  semigroup associated to $\sqrt L$. { When $n=1$ or 2,} we have the pointwise Poisson kernel estimate, hence $L^1-L^2$ and $L^2-L^\infty$ off-diagonal estimates with order  $ \frac{n}{2}+1$ and homogeneity $m=1$.   Hence the conclusion  of the second part in Proposition \ref{cor:Laplace} applies since $m=1$ and $\beta=-1$.
For $n\ge 3$, with the same numbers $p_{-}(L), p_{+}(L)$ as above, the resolvent estimate (H2) holds with $m=1$ and $p_{-}(L)<q\le r <p_{+}(L)$. Taking  $q<\frac{2n}{n+2}$, we use the first  item  in Corollary \ref{cor:MLtp2}, A]  when $n=3, 4$ and the third one when $n\ge 5$ to get the lower bound on $p$. For the upper bound, we use B] with $q> \frac{2n}{n-2}$ and find $\infty$ included if $n=3,4$, and the proposed value if $n\ge 5$.
\end{proof}
}

\begin{corollary} Let $H=L^2(\R^n)$.

A] Assume (H2) holds for $(2,q) $ with $ 2\le q$.
 Then $\mM_L^-$ extends to a bounded operator on $T^{p,2,m}(t^{\beta}dt dy)$ for $2<p<p_{L}$ with $p_{L}$ calculated as follows:
 \begin{enumerate}
  \item If $\frac{n}{mq}<1$ and $\beta\ge 1$, $p_{L}=\infty$ (and boundedness holds at $\infty$).
  \item If $\frac{n}{mq}<1$ and $-1<\beta<1$, {$p_{L}= \infty$ (and boundedness holds at $p=\infty$) if $p_{c}<1$ and $p_{L}= p'_{c}$ if $p_{c}\ge 1$.}
 \item  If $\frac{n}{mq}\ge 1$ then $\frac{1}{p_{L}}- \frac{1}{2}= \frac{mq}{n}(-\frac{1}{2})=  -\frac{mq}{2n}$.
\end{enumerate}

B] Assume  (H2) holds for $(q,2) $ with $q\le 2$.
 Then $\mM_L^-$ extends to a bounded operator on $T^{p,2,m}(t^{\beta}dt dy)$ for $p_{L}<p<2$ with $p_{L}= \frac{2n}{n+mq'}$ if $mq'\le n$ and  $p_{L}=p_{M_{q}}$ if $mq'>n$.
 \end{corollary}

This time, this follows from Proposition \ref{prop:adjoint} and Theorem \ref{thm:lql2absdual} where one finds the value of $p_{c}$, using the operators
\[ \mM_\alpha^- f (t) = \int\limits_t^\infty (s-t)^{\alpha - 1} L^\alpha e^{-(s-t)L}f(s)\,  ds\]
and the interpolation procedure of Corollary  { \ref{cor:MLtp2}}. Details are left to the reader.

{\flushleft{\sc Pascal Auscher}}\\
Univ. Paris-Sud, laboratoire de Math\'ematiques, UMR 8628, F-91405 {\sc Orsay}; CNRS, F-91405 {\sc Orsay}. \\
{\tt pascal.auscher@math.u-psud.fr}\\

{\flushleft{\sc Christoph Kriegler}}\\
Laboratoire de Math\'ematiques (UMR 6620),
Universit\'e Blaise-Pascal (Clermont-Ferrand 2),
Campus des C\'ezeaux,
63177 Aubi\`ere Cedex\\
{\tt christoph.kriegler@math.univ-bpclermont.fr}\\

{\flushleft{\sc Sylvie Monniaux}}\\
LATP-UMR 7353, FST Saint-J\'er\^ome,
Univ. Aix-Marseille, F-13397 {\sc Marseille} C\'edex 20.\\
{\tt sylvie.monniaux@univ-amu.fr}\\

{\flushleft{\sc Pierre Portal}}\\
Permanent Address:\\
Universit\'e Lille 1, Laboratoire Paul Painlev\'e, F-59655 {\sc Villeneuve d'Ascq}.\\
Current Address:\\
Australian National University, Mathematical Sciences Institute, John Dedman Building,
Acton ACT 0200, Australia.\\
{\tt pierre.portal@math.univ-lille1.fr}\\

\end{document}